\documentclass[10pt]{amsart}
\usepackage{latexsym, amsmath,amssymb}
\usepackage{amsmath}
\usepackage{amsfonts}
\usepackage{amssymb}
\usepackage{amscd}
\usepackage{amsthm}
\usepackage{amsbsy}
\usepackage{graphicx}
\usepackage{bm}

\renewcommand{\epsilon}{\varepsilon}
\newcommand{\newsection}[1]
{\subsection{#1}\setcounter{theorem}{0} \setcounter{equation}{0}
\par\noindent}

\newtheorem{theorem}{Theorem}

\newtheorem{lemma}[theorem]{Lemma}
\newtheorem{corr}[theorem]{Corollary}

\newtheorem{proposition}[theorem]{Proposition}
\newtheorem{deff}[theorem]{Definition}
\newtheorem{remark}[theorem]{Remark}
\newcommand{\bth}{\begin{theorem}}
\newcommand{\ble}{\begin{lemma}}
\newcommand{\bcor}{\begin{corr}}

\newcommand{\bdeff}{\begin{deff}}

\newcommand{\bprop}{\begin{proposition}}
\newcommand{\ele}{\end{lemma}}
\newcommand{\ecor}{\end{corr}}
\newcommand{\edeff}{\end{deff}}

\newcommand{\eprop}{\end{proposition}}

\newcommand{\cd}{\, \cdot\, }

\renewcommand{\Pi}{\varPi}

\renewcommand{\epsilon}{\varepsilon}

\newcommand{\Aut}{\text{Aut}(p)}

\begin{document}


\title[Improved restriction estimates for nonpostive curvature manifolds]{An improvement on eigenfunction restriction estimates for compact boundaryless Riemannian manifolds with nonpositive sectional curvature}
\thanks{The author would like to thank her advisor, Christopher Sogge, cordially for his generous help and unlimited patience.}

\author{Xuehua Chen}
\address{Johns Hopkins University, Baltimore, MD}
\address{Email address: xchen@math.jhu.edu}

\maketitle

\begin{abstract}
Let $(M,g)$ be an $n$-dimensional compact boundaryless Riemannian manifold with nonpositive sectional curvature, then our conclusion is that we can give improved estimates for the $L^p$ norms of the restrictions of eigenfunctions to smooth submanifolds of dimension $k$, for $p>\dfrac{2n}{n-1}$ when $k=n-1$ and $p>2$ when $k\leq n-2$, compared to the general results of Burq, G\'erard and Tzvetkov \cite{burq}. Earlier, B\'erard \cite{Berard} gave the same improvement for the case when $p=\infty$, for compact Riemannian manifolds without conjugate points for $n=2$, or with nonpositive sectional curvature for $n\geq3$ and $k=n-1$. In this paper, we give the improved estimates for $n=2$, the $L^p$ norms of the restrictions of eigenfunctions to geodesics. Our proof uses the fact that, the exponential map from any point in $x\in M$ is a universal covering map from $\mathbb{R}^2\backsimeq T_{x}M$ to $M$, which allows us to lift the calculations up to the universal cover $(\mathbb{R}^2,\tilde{g})$, where $\tilde{g}$ is the pullback of $g$ via the exponential map. Then we prove the main estimates by using the Hadamard parametrix for the wave equation on $(\mathbb{R}^2,\tilde{g})$, the stationary phase estimates, and the fact that the principal coefficient of the Hadamard parametrix is bounded, by observations of Sogge and Zelditch in \cite{SZ}. The improved estimates also work for $n\geq 3$, with $p>\frac{4k}{n-1}$. We can then get the full result by interpolation.
\end{abstract}

\newsection{Introduction}

Let $(M,g)$ be a compact, smooth $n$-dimensional boundaryless Riemannian manifold with nonpositive sectional curvature. Denote $\Delta_g$ the Laplace operator associated to the metric $g$, and $d_g(x,y)$ the geodesic distance between $x$ and $y$ associated with the metric $g$. We know that there exist $\lambda\geq0$ and $\phi_\lambda\in L^2(M)$ such that $-\Delta_g\phi_\lambda=\lambda^2\phi_\lambda$, and we call $\phi_\lambda$ an eigenfunction corresponding to the eigenvalue $\lambda$. Let $\{e_j(x)\}_{j\in\mathbb{N}}$ be an $L^2(M)$-orthonormal basis of eigenfunctions of $\sqrt{-\Delta_g}$, with eigenvalues $\{\lambda_j\}_{j\in\mathbb{N}}$, and $\{E_j(x)\}_{j\in\mathbb{N}}$ be the projections onto the $j$-th eigenspace, restricted to $\Sigma$, i.e. $E_jf(x)=e_j(x)\int_M e_j(y)f(y)dy$, for any $f\in L^2(M)$, $x\in\Sigma$. We may consider only the positive $\lambda$'s as we are interested in the asymptotic behavior of the eigenfunction projections. Our main Theorem is the following.

\bth\label{theorem1}
Let $(M,g)$ be a  compact smooth $n$-dimensional boundaryless Riemannian manifold with nonpositive curvature, and $\Sigma$ be an $k$-dimensional smooth submanifold on $M$. Let $\{E_j(x)\}_{j\in\mathbb{N}}$ be the projections onto the $j$-th eigenspace, restricted to $\Sigma$. Given any $f\in L^2(M)$, we have the following estimate:

When $k=n-1$,
\begin{equation}\label{1.1}
||\sum_{|\lambda_j-\lambda|\leq(\log\lambda)^{-1}}E_jf||_{L^p(\Sigma)}\lesssim\frac{\lambda^{\delta(p)}}{(\log \lambda)^{\frac{1}{2}}}||f||_{L^2(M)},\ \ \ \forall p>\dfrac{2n}{n-1};
\end{equation}

When $k\leq n-2$,
\begin{equation}\label{1.2}
||\sum_{|\lambda_j-\lambda|\leq(\log\lambda)^{-1}}E_jf||_{L^p(\Sigma)}\lesssim\frac{\lambda^{\delta(p)}}{(\log \lambda)^{\frac{1}{2}}}||f||_{L^2(M)},\ \ \ \forall p>2,
\end{equation}
where $\delta(p)=\frac{n-1}{2}-\frac{k}{p}$.
\end{theorem}
Note that we may assume that $(M,g)$ is also simply connected in the proof.

The following corollary is an immediate consequence of this theorem.
\begin{corr}
Let $(M,g)$ be a  compact smooth $n$-dimensional boundaryless Riemannian manifold with nonpositive curvature, and $\Sigma$ be an $k$-dimensional smooth submanifold on $M$. For any eigenfunction $\phi_\lambda$ of $\Delta_g$ s.t. $-\Delta_g\phi_\lambda=\lambda^2\phi_\lambda$, we have the following estimate:

When $k=n-1$,
\begin{equation}\label{1.3}
||\phi_\lambda||_{L^p(\Sigma)}\lesssim\frac{\lambda^{\delta(p)}}{(\log \lambda)^{\frac{1}{2}}}||\phi_\lambda||_{L^2(M)},\ \ \ \forall p>\dfrac{2n}{n-1};
\end{equation}

When $k\leq n-2$,
\begin{equation}\label{1.4}
||\phi_\lambda||_{L^p(\Sigma)}\lesssim\frac{\lambda^{\delta(p)}}{(\log \lambda)^{\frac{1}{2}}}||\phi_\lambda||_{L^2(M)},\ \ \ \forall p>2,
\end{equation}
where $\delta(p)=\frac{n-1}{2}-\frac{k}{p}$.
\end{corr}

In \cite{rez}, Reznikov achieved weaker estimates for hyperbolic surfaces, which inspired this current line of research. In \cite{burq}, Theorem 3, Burq, G\'erard and Tzvetkov showed that given any $k$-dimensional submanifold $\Sigma$ of an $n$-dimensional compact boundaryless manifold $M$, for any $p>\dfrac{2n}{n-1}$ when $k=n-1$ and for any $p>2$ when $k\leq n-2$, one has
\begin{equation}\label{1.5}
||\phi_\lambda||_{L^p(\Sigma)}\lesssim\lambda^{\delta(p)}||\phi_\lambda||_{L^2(M)},
\end{equation}
while for $p=\frac{2n}{n-1}$ when $k=n-1$ and for $p=2$ when $k=n-2$ one has
\begin{equation}\label{1.6}
||\phi_\lambda||_{L^p(\Sigma)}\lesssim\lambda^{\delta(p)}(\log\lambda)^{\frac{1}{2}}||\phi_\lambda||_{L^2(M)}.
\end{equation}

Later on, Hu improved the result at one end point in \cite{Hu}, so that one has \eqref{1.5} for $p=\frac{2n}{n-1}$ when $k=n-1$. It is very possible that one can also improve the result at the other end point, where $p=2$, $k=n-2$, so that we also have \eqref{1.5} there. Our Theorem \ref{theorem4.1} gives an improvement for \eqref{1.5} of $(\log\lambda)^{-\frac{1}{2}}$ for $p\geq2$ for certain small $k$'s (See the remark after Theorem \ref{theorem4.1}).

Note that their proof of Theorem 3 in \cite{burq} indicates that for any $f\in L^2(M)$,
\begin{equation}\label{1.7}
||\sum_{|\lambda_j-\lambda|<1}E_jf||_{L^p(\Sigma)}\lesssim\lambda^{\delta(p)}||f||_{L^2(M)},
\end{equation}
for any $p\geq\dfrac{2n}{n-1}$ when $k=n-1$ and $p\geq2$ when $k\leq n-2$ except that there is an extra $(\log\lambda)^{\frac{1}{2}}$ on the right hand side when $p=2$ and $k=n-2$. In the proof, they constructed $\chi_\lambda=\chi(\sqrt{-\Delta_g}-\lambda)$ from $L^2(M)$ to $L^p(\Sigma)$, where $\chi\in\mathcal{S}(\mathbb{R})$ such that $\chi(0)=1$, and showed that $\chi_\lambda(\chi_\lambda)^*$ is an operator from $L^p(\Sigma)$ to $L^{p'}(\Sigma)$ with norm $O(\lambda^{2\delta(p)})$. That means, there exists at least an $\epsilon>0$ such that \begin{equation}\label{1.8}
||\sum_{|\lambda_j-\lambda|<\epsilon}E_jf||_{L^p(\Sigma)}\lesssim\lambda^{\delta(p)}||f||_{L^2(M)}.
\end{equation}

The reason why \eqref{1.8} is true can be seen in this way. Considering the dual form of
\begin{equation}
||\chi(\lambda-\sqrt{-\Delta_g})f||_{L^p(\Sigma)}\lesssim\lambda^{\delta(p)}||f||_{L^2(M)},
\end{equation}
which says
\begin{equation}\label{1.9}
||\sum_j\chi(\lambda-\lambda_j)E_j^*g||_{L^2(M)}\lesssim\lambda^{\delta(p)}||g||_{L^p(\Sigma)},
\end{equation}
where $E_j^*$ is the conjugate operator of $E_j$ such that $E_j^*g(x)=e_j(x)\int_\Sigma e_j(y)g(y)dy$, for any $g\in L^2(\Sigma)$ and $x\in M$. There exists an $\epsilon>0$ such that $\chi(t)>\frac{1}{2}$ when $|t|<\epsilon$ because we assumed that $\chi(0)=1$. Therefore, the square of the left hand side of \eqref{1.9} is
\begin{equation}
\sum_{|\lambda-\lambda_j|<\epsilon}||\chi(\lambda-\lambda_j)E_j^*g||_{L^2(M)}^2+\sum_{|\lambda-\lambda_j|>\epsilon}||\chi(\lambda-\lambda_j)E_j^*g||_{L^2(M)}^2\geq\frac{1}{4}\sum_{|\lambda-\lambda_j|<\epsilon}||E_j^*g||_{L^2(M)}^2.
\end{equation}
That means
\begin{equation}
||\sum_{|\lambda-\lambda_j|<\epsilon}E_j^*g||_{L^2(M)}\lesssim\lambda^{\delta(p)}||g||_{L^p(\Sigma)},
\end{equation}
which is the dual version of \eqref{1.8}.

If we divide the interval $(\lambda-1,\lambda+1)$ into $\frac{1}{\epsilon}$ sub-intervals whose lengths are $2\epsilon$, and apply the last estimate $\frac{1}{\epsilon}$ times, we get \eqref{1.7}. Thinking in this way, our estimates \eqref{1.1} and \eqref{1.2} are equivalent to the estimates for
\begin{equation}
||\sum_{|\lambda_j-\lambda|<\epsilon\log^{-1}\lambda} E_j||_{L^2(M)\rightarrow L^p(\Sigma)},
\end{equation}
for some number $\epsilon>0$, which is equivalent to estimating
\begin{equation}
||\chi(T(\lambda-\sqrt{-\Delta_g}))||_{L^2(M)\rightarrow L^p(\Sigma)},
\end{equation}
for $T\approx\log^{-1}\lambda$.

The estimates \eqref{1.5} and \eqref{1.6} are sharp when

1. $k\leq n-2$, $M$ is the standard sphere $\mathbb{S}^n$ and $\Sigma$ is any submanifold of dimension $k$; or

2. $k=n-1$ and $2\leq p\leq\dfrac{2n}{n-1}$, $M$  is the standard sphere $\mathbb{S}^n$ and $\Sigma$ is any hypersurface containing a piece of geodesic.

It is natural to try to improve it on Riemannian manifolds with nonpositive curvature. Recently, Sogge and Zelditch in \cite{SZ} showed that for any 2-dimensional compact boundaryless Riemannian manifold with nonpositive curvature one has
\begin{equation}\label{1.10}
\sup_{\gamma\in\Pi}||\phi_\lambda||_{L^{p}(\gamma)}/||\phi_\lambda||_{L^2(M)}=o(\lambda^{\frac{1}{4}}),\ \ \ \text{for}\ 2\leq p<4,
\end{equation}
where $\Pi$ denotes the space of all unit-length geodesics in $M$.
\eqref{1.7} is sharp for any compact manifolds, in the sense that we fix the scale of the spectral projection (See proof in \cite{burq}). If we are allowed to consider a smaller scale of spectral projection, then our theorem \ref{theorem1} is an improvement of $\sqrt{\log\lambda}$ for \eqref{1.7}, with the extra assumption that $M$ has nonpositive curvature. The corollary is an improvement of \eqref{1.5}. Note that \eqref{1.3} and \eqref{1.10} improve \eqref{1.5} for the whole range of $p$ in dimension 2 except for $p=4$.

Theorem \ref{theorem1} is related to certain $L^p$-estimates for eigenfunctions. For example, for 2-dimensional Riemannian manifolds, Sogge showed in \cite{Sokakeya} that
\begin{equation}
||\phi_\lambda||_{L^p(M)}/||\phi_\lambda||_{L^2(M)}=o(\lambda^{\frac{1}{2}(\frac{1}{2}-\frac{1}{p})})
\end{equation}
for some $2<p<6$ if and only if
\begin{equation}
\sup_{\gamma\in\Pi}||\phi_\lambda||_{L^2(\gamma)}/||\phi_\lambda||_{L^2(M)}=o(\lambda^{\frac{1}{4}}).
\end{equation}
This indicates relations between the restriction theorem and the $L^p$-estimates for eigenfunctions in \cite{soggeest} by Sogge, which showed that for any compact Riemannian manifold of dimension $n$, one has
\begin{equation}
||\phi_\lambda||_{L^p(M)}\lesssim\lambda^{\frac{n-1}{2}(\frac{1}{2}-\frac{1}{p})}||\phi_\lambda||_{L^2(M)},\ \ \ \text{for}\ 2\leq p\leq\dfrac{2(n+1)}{n-1},
\end{equation}
and
\begin{equation}\label{1.12}
||\phi_\lambda||_{L^p(M)}\lesssim\lambda^{n(\frac{1}{2}-\frac{1}{p})-\frac{1}{2}}||\phi_\lambda||_{L^2(M)},\ \ \ \text{for}\ \dfrac{2(n+1)}{n-1}\leq p\leq\infty.
\end{equation}

There have been several results showing that \eqref{1.12} can be improved for $p>\dfrac{2(n+1)}{n-1}$ (see \cite{stz} and \cite{soggezelditch}) to bounds of the form $||\phi_\lambda||_{L^p(M)}/||\phi_\lambda||_{L^2(M)}=o(\lambda^{n(\frac{1}{2}-\frac{1}{p})-\frac{1}{2}})$ for fixed $p>6$. Recently, Hassell and Tacey \cite{HT}, following B\'erard's \cite{Berard} estimate for $p=\infty$, showed that for fixed $p>6$, this ratio is $O(\lambda^{n(\frac{1}{2}-\frac{1}{p})-\frac{1}{2}}/\sqrt{\log\lambda})$ on Riemannian manifolds with constant negative curvature, which inspired our work.

\newsection{Set up of the proof of the improved restriction theorem}

Let us first analyze the situation for any dimension $n$, which we will use in Section 4.

Take a real-valued multiplier operator $\chi\in\mathcal{S}(\mathbb{R})$ such that $\chi(0)=1$, and $\hat{\chi}(t)=0$ if $|t|\geq\frac{1}{2}$.
Let $\rho=\chi^2$, then $\hat{\rho}(t)=0$ if $|t|\geq1$. Here, $\hat{\chi}$ is the Fourier Transform of $\chi$. Same notations in the following.

For some number $T$, which will be determined later, and is approximately $\log\lambda$, we have $\chi(T(\lambda-\sqrt{-\Delta_g}))\varphi_\lambda=\varphi_\lambda$. The theorem is proved if we can show that for any $f\in L^2(M)$,
\begin{equation}\label{2.1}
||\chi^\lambda_Tf||_{L^p(\Sigma)}\lesssim\frac{\lambda^{\delta(p)}}{(\log\lambda)^\frac{1}{2}}||f||_{L^2(M)},
\end{equation}
where $\chi^\lambda_T=\chi(T(\lambda-\sqrt{-\Delta_g}))$ is an operator from $L^2(M)$ to $L^p(\Sigma)$.

This is equivalent to for any $g\in L^{p'}(\Sigma)$,
\begin{equation}\label{2.2}
||\chi^\lambda_T(\chi^\lambda_T)^* g||_{L^p(\Sigma)}\lesssim\frac{\lambda^{2\delta(p)}}{\log\lambda}||g||_{L^{p'}(\Sigma)},
\end{equation}
where $p'$ is the conjugate number of $p$ such that $\frac{1}{p}+\frac{1}{p'}=1$. and $(\chi_T^\lambda)^*$ is the conjugate of $\chi_T^\lambda$, which maps $L^{p'}(\Sigma)$ into $L^2(M)$.

If $\{e_j(x)\}_{j\in\mathbb{N}}$ is an $L^2(M)$ orthonormal basis of eigenfunctions of $\sqrt{-\Delta_g}$, with eigenvalues $\{\lambda_j\}_{j\in\mathbb{N}}$, and $\{E_j(x)\}_{j\in\mathbb{N}}$ is the projections onto the $j$-th eigenspace restricted to $\Sigma$, then $I|_\Sigma=\sum_{j\in\mathbb{N}}E_j$, and $\sqrt{-\Delta_g}|_\Sigma=\sum_{j\in\mathbb{N}}\lambda_jE_j$. If we set $\rho^\lambda_T=\rho(T(\lambda-\sqrt{-\Delta_g})): L^2(M)\rightarrow L^p(\Sigma)$, then the kernel of $\chi_T^\lambda(\chi_T^\lambda)^*$ is the kernel of $\rho_T^\lambda$, which is restricted to $\Sigma\times\Sigma$. That can be seen from
\begin{equation}
\chi^\lambda_T f(x)=\sum_{j\in\mathbb{N}}\chi(T(\lambda-\lambda_j))e_j(x)\int_Me_j(y)f(y)dy,\ \ \forall f\in L^2(M),
\end{equation}
and
\begin{equation}
(\chi^\lambda_T)^* g(x)=\sum_{j\in\mathbb{N}}\chi(T(\lambda-\lambda_j))e_j(x)\int_\Sigma e_j(y)g(y)dy,\ \ \forall g\in L^{p'}(\Sigma).
\end{equation}
Therefore,
\begin{equation}
\begin{split}
\chi^\lambda_T(\chi^\lambda_T)^* g(x) & =\sum_{i,j\in\mathbb{N}}\chi(T(\lambda-\lambda_i))\chi(T(\lambda-\lambda_j))e_j(x)\int_Me_j(y)e_i(y)\int_\Sigma e_i(z)g(z)dzdy\\
& =\sum_{j\in\mathbb{N}}\chi(T(\lambda-\lambda_j))^2e_j(x)\int_\Sigma e_j(z)g(z)dz\\
& =\sum_{j\in\mathbb{N}}\rho(T(\lambda-\lambda_j))e_j(x)\int_\Sigma e_j(z)g(z)dz.
\end{split}
\end{equation}

On the other hand,
\begin{equation}
\begin{split}
\rho^\lambda_T & =\sum_{j\in\mathbb{N}}\rho(T(\lambda-\lambda_j))E_j \\
& =\sum_{j\in\mathbb{N}}\frac{1}{2\pi}\int_{-1}^1\hat{\rho}(t)e^{it[T(\lambda-\lambda_j)]}E_jdt\\
& =\sum_{j\in\mathbb{N}}\frac{1}{2\pi T}\int_{-T}^T\hat{\rho}(\frac{t}{T})e^{it(\lambda-\lambda_j)}E_jdt\\
& =\frac{1}{2\pi T}\int_{-T}^T\hat{\rho}(\frac{t}{T})e^{it(\lambda-\sqrt{-\Delta_g})}dt\\
& =\frac{1}{\pi T}\int_{-T}^T\hat{\rho}(\frac{t}{T})\cos(t\sqrt{-\Delta_g})e^{it\lambda}dt-\rho(T(\lambda+\sqrt{-\Delta_g}))
\end{split}
\end{equation}
Here, $\rho(T(\lambda+\sqrt{-\Delta_g}))$ is an operator whose kernel is $O(\lambda^{-N})$, for any $N\in\mathbb{N}$, so that we only have to estimate the first term. We are not going to emphasize the restriction to $\Sigma$ until we get to the point when we take the $L^p$ norm on $\Sigma$.

Denote the kernel of $\cos (t\sqrt{-\Delta_g})$ as $\cos(t\sqrt{-\Delta_g})(x,y)$, for $x,y\in M$, then $\forall g\in L^{p'}(\Sigma)$,
\begin{equation}\label{2.3}
\chi^\lambda_T(\chi^\lambda_T)^* g(x)=\frac{1}{\pi T}\int_\Sigma\int_{-T}^T\hat{\rho}(\frac{t}{T})\cos(t\sqrt{-\Delta_g})(x,y)e^{it\lambda}g(y)dtdy+O(1).
\end{equation}

Take the $L^p(\Sigma)$ norm on both sides,
\begin{equation}\label{2.4}
||\chi^\lambda_T(\chi^\lambda_T)^*g||_{L^p(\Sigma)}\leq\frac{1}{\pi T}(\int_\Sigma|\int_\Sigma\int_{-T}^T\hat{\rho}(\frac{t}{T})\cos(t\sqrt{-\Delta_g})(x,y)e^{it\lambda}g(y)dtdy|^pdx)^{1/p}+O(1).
\end{equation}

We are going to use Young's inequality (see \cite{soggebook}), with $\frac{1}{r}=1-[(1-\frac{1}{p})-\frac{1}{p}]=\frac{2}{p}$, and
\begin{equation}\label{2.5}
K(x,y)=\frac{1}{\pi T}\int_{-T}^T\hat{\rho}(\frac{t}{T})(\cos t\sqrt{-\Delta_g})(x,y))e^{it\lambda}dt.
\end{equation}
Denote $K$ as the operator with the kernel $K(x,y)$ from now on.\footnote{The definition of $K(x,y)$ may be changed in this paper, but we always call $K$ the corresponding operator with the kernel $K(x,y)$.}

Since $K(x,y)$ is symmetric in $x$ and $y$, once we have
\begin{equation}\label{2.6}
\sup_{x\in\Sigma}||K(x,\cdot)||_{L^r(\Sigma)}\lesssim \frac{\lambda^{2\delta(p)}}{\log\lambda},
\end{equation}
where $r=p/2$, then by Young's inequality, the theorem is proved.

We can use the same argument as in \cite{SZ} to lift the manifold to $\mathbb{R}^n$. As stated in Theorem IV.1.3 in \cite{IS}, for $(M,g)$ has non-positive curvature, considering $x$ to be a fixed point on $\Sigma$, there exists a universal covering map $p=\exp_{x}:\mathbb{R}^n\rightarrow M$. In this way, $(M,g)$ is lifted to $(\mathbb{R}^n,\tilde{g})$, with the metric $\tilde{g}=(\exp_x)^*g$ being the pullback of $g$ via $\exp_{x}$. $\tilde{g}$ is a complete Riemannian metric on $\mathbb{R}^n$. Define an automorphism for $(\mathbb{R}^n,\tilde{g})$, $\alpha: \mathbb{R}^n\rightarrow\mathbb{R}^n$, to be a deck transformation if
\begin{equation}\nonumber
p\circ\alpha=p,
\end{equation}
when we shall write $\alpha\in\Aut$. If $\tilde{x}\in\mathbb{R}^n$ and $\alpha\in\Aut$, let us call $\alpha(\tilde{x})$ the translate of $\tilde{x}$ by $\alpha$, then we call a simply connected set $D\subset\mathbb{R}^n$ a fundamental domain of our universal covering $p$ if every point in $\mathbb{R}^n$ is the translate of exactly one point in $D$. We can then think of our submanifold $\Sigma$ both as one in $(M,g)$ and one in the fundamental domain which is of the same form. Likewise, a function $f(x)$ in $M$ is uniquely identified by one $f_D(\tilde{x})$ on $D$ if we set $f_D(\tilde{x})=f(x)$, where $\tilde{x}$ is the unique point in $D\cap p^{-1}(x)$. Using $f_D$ we can define a "periodic extension", $\tilde{f}$, of $f$ to $\mathbb{R}^n$ by defining $\tilde{f}(\tilde{y})$ to be equal to $f_D(\tilde{x})$ if $\tilde{x}=\tilde{y}$ modulo $\Aut$, i.e. if $(\tilde{x},\alpha)\in D\times\Aut$ are the unique pair so that $\tilde{y}=\alpha(\tilde{x})$.

In this setting, we shall exploit the relationship between solutions
of the wave equation on $(M,g)$ of the form
\begin{equation}\label{2.20}
\begin{cases}
(\partial^2_t-\Delta_g)u(t,x)=0, \quad (t,x)\in {\mathbb R}_+\times M
\\
u(0,\cd)=f, \, \, \partial_t u(0,\cd)=0,
\end{cases}
\end{equation}
and certain ones on $({\mathbb R}^n,\tilde g)$
\begin{equation}\label{2.21}
\begin{cases}
(\partial^2_t-\Delta_{\tilde g})\tilde u(t,\tilde x)=0, \quad (t,\tilde x)\in {\mathbb R}_+\times
\mathbb{R}^n
\\
\tilde u(0,\cd)=\tilde f, \, \, \partial_t \tilde u(0,\cd)=0.
\end{cases}
\end{equation}
If $(f(x),0)$ is the Cauchy data in \eqref{2.20} and $(\tilde{f}(\tilde{x}),0)$ is the periodic extension to $(\mathbb{R}^n,\tilde{g})$, then the solution $\tilde{u}(t,\tilde{x})$ to \eqref{2.21} must be a periodic function of $\tilde{x}$ since $\tilde{g}$ is the pullback of $g$ via $p$ and $p\circ\alpha=p$. As a result, we have that the solution to \eqref{2.20} must satisfy $u(t,x)=\tilde{u}(t,\tilde{x})$ if $\tilde{x}\in D$ and $p(\tilde{x})=x$. Thus, periodic solutions to \eqref{2.21} correspond uniquely to solutions of \eqref{2.20}. Note that $u(t,x)=\bigl(\cos (t\sqrt{-\Delta_g})f\bigr)(x)$ is the solution of \eqref{2.20}, so that
\begin{equation}
\cos(t\sqrt{-\Delta_g})(x,y)=\sum_{\alpha\in\Aut}\cos (t\sqrt{-\Delta_{\tilde{g}}})(\tilde{x},\alpha(\tilde{y})),
\end{equation}
if $\tilde{x}$ and $\tilde{y}$ are the unique points in $D$ for which $p(\tilde{x})=x$ and $p(\tilde{y})=y$.

\newsection{Proof of the improved restriction theorem, for $n=2$}

While we can prove Theorem \ref{theorem1} for any dimension $n$, we will prove the case when $n=2$ first separately, as it is the simplest case, and does not involve interpolation or various sub-dimensions. Here is what it says.
\begin{theorem}\label{theorem2}
Let $(M,g)$ be a compact smooth boudaryless Riemannian surface with nonpositive curvature, and $\gamma$ be a smooth curve with finite length, then for any $f\in L^2(M)$, we have the following estimate
\begin{equation}\label{3.1}
||\sum_{|\lambda_j-\lambda|<(\log\lambda)^{-1}}E_j f||_{L^p(\gamma)}\lesssim\frac{\lambda^{\frac{1}{2}-\frac{1}{p}}}{(\log \lambda)^{\frac{1}{2}}}||f||_{L^2(M)},\ \ \ \forall p>4.
\end{equation}
\end{theorem}

We will prove Theorem~\ref{theorem2} by the end of this section. By a partition of unity, we can assume that we fix $x$ to be the mid-point of $\gamma$, and parametrize $\gamma$ by its arc length centered at $x$ so that
\begin{equation}\label{2.22}
\gamma=\gamma[-1,1]\ \ \ \text{and}\ \ \  \gamma(0)=x,
\end{equation}
and we may assume that the geodesic distance between any $x$ and $y\in\gamma$ is comparable to the arc length between them on $\gamma$.

We need to estimate the $L^r(\gamma)$ norm of
\begin{equation}\label{2.7}
\int_{-T}^T\hat{\rho}(\frac{t}{T})(\cos t\sqrt{-\Delta_g})(x,y)e^{it\lambda}dt=\sum_{\alpha\in \Aut}\int_{-T}^T\hat{\rho}(\frac{t}{T})(\cos t\sqrt{-\Delta_{\tilde{g}}})(\tilde{x},\alpha(\tilde{y}))e^{it\lambda}dt.
\end{equation}

We should have the following estimates:

Up to an error of $O(\lambda^{-1})\exp(O(d_{\tilde{g}}(\tilde{x},\tilde{y})))+O(e^{dT})$ or $O(\lambda^{-1})\exp(O(d_{\tilde{g}}(\tilde{x},\alpha(\tilde{y}))))+O(e^{dT})$ respectively,
\begin{equation}\label{2.9}
\int_{-T}^T\hat{\rho}(\frac{t}{T})(\cos t\sqrt{-\Delta_{\tilde{g}}})(\tilde{x},\tilde{y})e^{it\lambda}dt=O(\lambda)\ \ \  \textit{when}\ \ \ d_{\tilde{g}}(\tilde{x},\tilde{y})<\frac{1}{\lambda},
\end{equation}

\begin{equation}\label{2.10}\int_{-T}^T\hat{\rho}(\frac{t}{T})(\cos t\sqrt{-\Delta_{\tilde{g}}})(\tilde{x},\tilde{y})e^{it\lambda}dt=O((\frac{\lambda}{d_{\tilde{g}}(\tilde{x},\tilde{y})})^{1/2})\ \ \  \textit{when}\ \ \ d_{\tilde{g}}(\tilde{x},\tilde{y})\geq\frac{1}{\lambda},
\end{equation}

\begin{equation}\label{2.11}\alpha\neq Id, \int_{-T}^T\hat{\rho}(\frac{t}{T})(\cos t\sqrt{-\Delta_{\tilde{g}}})(\tilde{x},\alpha(\tilde{y}))e^{it\lambda}dt=O((\frac{\lambda}{d_{\tilde{g}}(\tilde{x},\alpha(\tilde{y}))})^\frac{1}{2})
\end{equation}

To prove the above estimates, we need the following lemma.

\begin{lemma}\label{lemma4.3}
Assume that $w(\tilde{x},\tilde{x}')$ is a smooth function from $\mathbb{R}^n\times\mathbb{R}^n$ to $\mathbb{R}^n$, and $\Theta\in\mathbb{S}^{n-1}$, then
\begin{equation}\label{4.2}
\int_{\mathbb{S}^{n-1}}e^{iw(\tilde{x},\tilde{x}')\cdot\Theta}d\Theta=\sqrt{2\pi}^{n-1}\sum_{\pm}\dfrac{e^{\pm i|w(\tilde{x},\tilde{x}')|}}{|w(\tilde{x},\tilde{x}')|^{\frac{n-1}{2}}}+O(|w(\tilde{x},\tilde{x}')|^{-\frac{n-1}{2}-1}),
\end{equation}
when $|w(\tilde{x},\tilde{x}')|\geq1$.
\end{lemma}
The proof can be found in Chapter 1 in \cite{soggebook}.

Let us return to estimating the kernel $K(x,y)$. Applying the Hadamard Parametrix,
\begin{equation}\label{2.14}
\cos(t\sqrt{-\Delta_{\tilde{g}}})(\tilde{x},\alpha(\tilde{y}))=\dfrac{w_0(\tilde{x},\alpha(\tilde{y}))}{(2\pi)^n}\sum_{\pm}\int_{\mathbb{R}^n}e^{i\Phi(\tilde{x},\alpha(\tilde{y}))\cdot\xi\pm
it|\xi|}d\xi+\sum_{\nu=1}^N w_\nu(\tilde{x},\alpha(\tilde{y}))\mathcal{E}_\nu(t,d_{\tilde{g}}(\tilde{x},\alpha(\tilde{y})))+R_N(t,\tilde{x},\alpha(\tilde{y})),
\end{equation}
where
$|\Phi(\tilde{x},\alpha(\tilde{y}))|=d_{\tilde{g}}(\tilde{x},\alpha(\tilde{y}))$, $\mathcal{E}_\nu,\nu=1,2,3,...$ are defined recursively by
$2\mathcal{E}_\nu(t,r)=-t\int_0^t\mathcal{E}_{\nu-1}(s,r)ds$, where $\mathcal{E}_0(t,x)=(2\pi)^{-n}\int_{\mathbb{R}^n}e^{ix\cdot\xi}\cos(t|\xi|)d\xi$\footnote{Since $\mathcal{E}\nu(t,x)$ is invariant under the same radius, we consider $\mathcal{E}\nu(t,x)=\mathcal{E}\nu(t,|x|)$.}, and $w_\nu(\tilde{x},\alpha(\tilde{y}))$ equals some constant times $u_\nu(\tilde{x},\alpha(\tilde{y}))$ that satisfies:
\begin{equation}\label{3.2}
\begin{cases}
u_0(\tilde{x},\alpha(\tilde{y}))=\Theta^{-\frac{1}{2}}(\alpha(\tilde{y}))
\\
u_{\nu+1}(\tilde{x},\alpha(\tilde{y}))=\Theta(\alpha(\tilde{y}))\int_0^1s^\nu\Theta^{\frac{1}{2}}(\tilde{x}_s)\Delta_{\tilde{g}}u_\nu(\tilde{x},\tilde{x}_s)ds,\ \ \ \nu\geq0.
\end{cases}
\end{equation}
where $\Theta(\alpha(\tilde{y}))=(\det g_{ij}(\alpha(\tilde{y})))^{\frac{1}{2}}$, and $(\tilde{x}_s)_{s\in[0,1]}$ is the minimizing geodesic from $\tilde{x}$ to $\alpha(\tilde{y})$ parametrized proportionally to arc length.
 (see \cite{Berard} and \cite{SZ})

First note that for $N\geq n+\frac{3}{2}$, by using the energy estimates (see \cite{let} Theorem 3.1.5), one can show that $|R_N(t,\tilde{x},\alpha(\tilde{y}))|=O(e^{dt})$, for some constant $d>0$, so that it is small compared to the first $N$ terms.

\begin{theorem}\label{theorem3.1}
Given an $n$-dimensional compact Riemannian manifold $(M,g)$ with nonpositve curvature, and let $(\mathbb{R}^n,\tilde{g})$ be the universal covering of $(M,g)$. Then if $N\geq n+\frac{3}{2}$, in local coordinates,
\begin{equation}
(\cos t\sqrt{-\Delta_{\tilde{g}}})f(\tilde{x})=\int K_N(t,\tilde{x};\tilde{y})f(\tilde{y})dV_{\tilde{g}}(\tilde{y})+\int R_N(t,\tilde{x};\tilde{y})f(\tilde{y})dV_{\tilde{g}}(\tilde{y}),
\end{equation}
where
\begin{equation}
K_N(t,\tilde{x};\tilde{y})=\sum_{\nu=0}^N w_\nu(\tilde{x},\tilde{y})\mathcal{E}_\nu(t,d_{\tilde{g}}(\tilde{x},\tilde{y})),
\end{equation}
with the remainder kernel $R_N$ satisfying
\begin{equation}\label{3.5}
|R_N(t,\tilde{x};\tilde{y}))|=O(e^{dt}).
\end{equation}
for some number $d>0$.
\end{theorem}
This comes from Equation (42) in \cite{Berard}. The proof can be found in \cite{Berard}.

By this theorem,
\begin{equation}\label{3.14}
\int_{-T}^T|R_N(t,\tilde{x},\alpha(\tilde{y}))|dt\leq C\int_0^Te^{dt}dt=O(e^{dT}).
\end{equation}
Moreover, for $\nu=1,2,3,...$, we have the following estimate for $\mathcal{E}_\nu(t,r)$.

\begin{theorem}For $\nu=0,1,2,...$ and $\mathcal{E}_\nu(t,r)$ defined above, we have
\begin{equation}\label{3.22}|\int \hat{\rho}(t)e^{it\lambda}\mathcal{E}_\nu(t,r)dt|=O(\lambda^{n-1-2\nu}),\ \ \ \lambda\geq1
\end{equation}
\end{theorem}
\begin{proof}
Recall that
\begin{equation}
\mathcal{E}_0(t,r)=\dfrac{H(t)}{(2\pi)^n}\int_{\mathbb{R}^n}e^{i\Phi(\tilde{x},\tilde{y})\cdot\xi}\cos t|\xi|d\xi,
\end{equation}
so that
\begin{equation}
\begin{split}
|\int \hat{\rho}(t)e^{it\lambda}\mathcal{E}_0(t,r)dt| & =|\frac{1}{2(2\pi)^n}\int\int_{\mathbb{R}^n}\hat{\rho}(t)e^{it(\lambda\pm|\xi|)+i\Phi(\tilde{x},\tilde{y})\cdot\xi}d\xi dt|\\
& \approx |\int_{\mathbb{R}^n}[\rho(\lambda+|\xi|)+\rho(\lambda-|\xi|)]e^{i\Phi(\tilde{x},\tilde{y})\cdot\xi}d\xi|\\
& \leq \int_{\mathbb{R}^n}|\rho(\lambda+|\xi|)|+|\rho(\lambda-|\xi|)|d\xi\\
& =O(\lambda^{n-1}).
\end{split}
\end{equation}
By the definition of $\mathcal{E}_\nu$ such that $\dfrac{\partial \mathcal{E}_\nu}{\partial t}=\frac{t}{2}\mathcal{E}_{\nu-1}$ and integrate by parts, we get that for any $\nu=1,2,3,...$,
\begin{equation}
\int \hat{\rho}(t)e^{it\lambda}\mathcal{E}_\nu(t,r)dt=O(\lambda^{n-1-2\nu}).
\end{equation}
\end{proof}

The following theorem has been shown by B\'erard in \cite{Berard} about the size of the coefficients $u_k(\tilde{x},\tilde{y})$.

\begin{theorem}\label{theorem3.3}
Let $(M,g)$ be a compact $n$-dimensional Riemannian manifold and let $\sigma$ be its sectional curvature (hence, there is a number $\Gamma$ such that $-\Gamma^2\leq\sigma$). Assume that either

1.$n=2$, and $M$ does not have conjugate points;

or

2.$-\Gamma^2\leq\sigma\leq0$; i.e. $M$ has nonpositive sectional curvature.

Let $(\mathbb{R}^n,\tilde{g})$ be the universal covering of $(M,g)$, and let $\tilde{u}_\nu$, $\nu=0,1,2,...$ be defined by the relations \eqref{3.2}, then for any integers $l$ and $\nu$
\begin{equation}
\Delta_{\tilde{g}}^l\tilde{u}_\nu(\tilde{x},\tilde{y})=O(\exp(O(d_{\tilde{g}}(\tilde{x},\tilde{y})))).
\end{equation}
\end{theorem}

The proof can be found in \cite{Berard} Appendix: Growth of the Functions $u_k(x,y)$.

Since $w_\nu(\tilde{x},\alpha(\tilde{y}))$ is a constant times $\tilde{u}_\nu(\tilde{x},\alpha(\tilde{y}))$, this theorem tells us that $|w_\nu(\tilde{x},\alpha(\tilde{y}))|=O(\exp(c_\nu d_{\tilde{g}}(\tilde{x},\alpha(\tilde{y}))))$, for some constant $c_\nu$ depending on $\nu$.

Moreover, denote that $\psi(t)=\hat{\rho}(\frac{t}{T})$, and $\tilde{\psi}$ is the inverse Fourier Transform of $\psi$, we have $\tilde{\psi}\in\mathcal{S}(\mathbb{R})$ such that
\begin{equation}
|\tilde{\psi}(t)|\leq T(1+T|t|)^{-N},\ \ \ \text{for all}\ N\in\mathbb{N}.
\end{equation}
Therefore,
\begin{equation}\label{3.31}
\begin{split}
&\sum_{\nu=1}^N |w_\nu(\tilde{x},\alpha(\tilde{y}))\int_{-T}^T\hat{\rho}(\frac{t}{T})e^{it\lambda}\mathcal{E}_\nu(t,d_{\tilde{g}}(\tilde{x},\alpha(\tilde{y})))dt|\\
=&\sum_{\nu=1}^N O(T(T\lambda)^{n-1-2\nu}\exp(c_\nu d_{\tilde{g}}(\tilde{x},\alpha(\tilde{y}))))\\
=&O(T^{n-2}\lambda^{n-3}\exp(C_N d_{\tilde{g}}(\tilde{x},\alpha(\tilde{y})))),
\end{split}
\end{equation}
for some $C_N$ depending on $c_1, c_2,..., c_{N-1}$.

All in all, taking $n=2$, and disregarding the integral of the remainder kernel,
\begin{equation}\label{3.32}
\begin{split}
&|\int_{-T}^T\hat{\rho}(\frac{t}{T})\cos(t\sqrt{-\Delta_{\tilde{g}}})(\tilde{x},\alpha(\tilde{y}))e^{it\lambda}dt|\\
=&|\int_{-T}^T\hat{\rho}(\frac{t}{T})\dfrac{w_0(\tilde{x},\alpha(\tilde{y}))}{4\pi^2}\sum_{\pm}\int_{\mathbb{R}^2}e^{i\Phi(\tilde{x},\alpha(\tilde{y}))\cdot\xi\pm
it|\xi|}e^{it\lambda}d\xi dt|+O(\lambda^{-1}\exp(C_N d_{\tilde{g}}(\tilde{x},\alpha(\tilde{y})))).
\end{split}
\end{equation}

On the other hand, $w_0(\tilde{x},\tilde{y})$ has a better estimate. By applying G\"unther's Comparison Theorem \cite{Gu}, with the assumption of nonpositive curvature, we can show that $|w_0(\tilde{x},\tilde{y})|=O(1)$. The proof is given by Sogge and Zelditch in \cite{SZ} for $n=2$. Let's see the case for any dimension $n$. In  the geodesic polar coordinates we are
using, $t\Theta$, $t>0$, $\Theta\in \mathbb{S}^{n-1}$, for $(\mathbb{R}^n, \tilde g)$, the metric $\tilde g$ takes the form
\begin{equation}
ds^2=dt^2+\mathcal{A}^2(t,\Theta)\, d\Theta^2,
\end{equation}
where we may assume that $\mathcal{A}(t,\Theta)>0$ for $t>0$. Consequently, the volume element in these coordinates is given by
\begin{equation}\label{3.7}
dV_g(t,\theta)=\mathcal{A}(t,\Theta)\, dt d\Theta,
\end{equation}
and by G\"unther's \cite{Gu} comparison theorem if the curvature of $(M,g)$, which is the same as that of $(\mathbb{R}^n, \tilde g)$ is nonpositive, we have
\begin{equation}
\mathcal{A}(t,\theta)\ge t^{n-1},
\end{equation}
where $t^{n-1}$ is the volume element of the Euclidean space.
While in geodesic normal coordinates about $x$, we have $$w_0(x,y)=\bigl(\, \text{det } g_{ij}(y)\, \bigr)^{-\frac14},$$ (see \cite{Berard}, \cite{Had} or \S2.4 in \cite{let}).  If $y$ has geodesic polar coordinates $(t,\Theta)$ about $x$, then $t=d_{\tilde g}(x,y)$, so that $w_0(x,y)=\sqrt{t^{n-1}/{\mathcal A}(t,\Theta)}\leq1$.

Therefore,
\begin{equation}\label{2.18}
\begin{split}
|\sum_{\pm}\int_{\mathbb{R}^2}\int_{-T}^Te^{i\Phi(\tilde{x},\tilde{y})\cdot\xi\pm it|\xi|+it\lambda}\hat{\rho}(\frac{t}{T})dtd\xi|= & |\int_{\mathbb{R}^2}e^{i\Phi(\tilde{x},\tilde{y})\cdot\xi}(\tilde{\psi}(\lambda+|\xi|)+\tilde{\psi}(\lambda-|\xi|))d\xi| \\
\leq & \int_{\mathbb{R}^2}|\tilde{\psi}(\lambda+|\xi|)|d\xi+\int_{\mathbb{R}^2}|\tilde{\psi}(\lambda-|\xi|)|d\xi
\end{split}
\end{equation}

Note that $\tilde{\psi}(\lambda+|\xi|)=O(T(1+\lambda+|\xi|)^{-N})$, for any $N\in\mathbb{N}$, so $\int_{\mathbb{R}^2}\tilde{\psi}(\lambda+|\xi|)d\xi$ can be arbitrarily small, while $\tilde{\psi}(\lambda-|\xi|)=O(T(1+T|\lambda-|\xi||)^{-N})$, for any $N\in\mathbb{N}$, so that $\int_{\mathbb{R}^2}|\tilde{\psi}(\lambda-|\xi|)|d\xi\lesssim T\int_{\lambda-1\leq|\xi|\leq\lambda+1}(1+T|\lambda-|\xi||)^{-N}d\xi=O(\lambda)$, provided that $\lambda\geq1$. So
\begin{equation}\label{2.19}
\int_{-T}^T\hat{\rho}(\frac{t}{T})(\cos t\sqrt{-\Delta_{\tilde{g}}})(\tilde{x},\tilde{y})e^{it\lambda}dt=O(\lambda)+O(\lambda^{-1}\exp(C_N d_{\tilde{g}}(\tilde{x},\tilde{y}))),
\end{equation}
disregarding the integral of the remainder kernel.

However, this estimate can be improved when $d_{\tilde{g}}(\tilde{x},\tilde{y})\geq\frac{1}{\lambda}$.

As we can see, the main term of
\begin{equation}
\cos(t\sqrt{-\Delta_{\tilde{g}}})(\tilde{x},\tilde{y})=\dfrac{w_0(\tilde{x},\tilde{y})}{4\pi^2}\sum_{\pm}\int_{\mathbb{R}^2}e^{i\Phi(\tilde{x},\tilde{y})\cdot\xi\pm
it|\xi|}d\xi+\sum_{\nu=1}^N w_\nu(\tilde{x},\tilde{y})\mathcal{E}_\nu(t,d_{\tilde{g}}(\tilde{x},\tilde{y}))+R_N(t,\tilde{x},\tilde{y})
\end{equation}
comes from the first term, and the corresponding term in $\int_{-T}^T\hat{\rho}(\frac{t}{T})(\cos t\sqrt{-\Delta_{\tilde{g}}})(\tilde{x},\tilde{y})e^{it\lambda}dt$ is bounded by
\begin{equation}
C|\sum_\pm\int_{-T}^T\int_{\mathbb{R}^2}\hat{\rho}(\frac{t}{T})e^{i\Phi(\tilde{x},\tilde{y})\cdot\xi\pm it|\xi|}e^{it\lambda}dtd\xi|=  C|\sum_\pm\int_{-T}^T\int_0^\infty\int_0^{2\pi}\hat{\rho}(\frac{t}{T})e^{ir\Phi(\tilde{x},\tilde{y})\cdot\Theta\pm itr+it\lambda}rdtdrd\theta|.
\end{equation}

Integrate with respect to $t$ first, then the quantity above is bounded by a constant times

\begin{equation}
\sum_\pm\int_0^\infty\int_0^{2\pi}\tilde{\psi}(\lambda\pm r)e^{ir\Phi(\tilde{x},\tilde{y})\cdot\Theta}rd\theta dr.
\end{equation}

Because $\tilde{\psi}(\lambda\pm r)\lesssim T(1+T|\lambda\pm r|)^{-N}$ for any $N>0$, the term with $\tilde{\psi}(\lambda+r)$ in the sum is $O(1)$, while the other term with $\tilde{\psi}(\lambda-r)$ is significant only when $r$ is comparable to $\lambda$, say, $c_1\lambda<r<c_2\lambda$ for some constants $c_1$ and $c_2$. In this case, as we assumed that $d_{\tilde{g}}(\tilde{x},\tilde{y})\geq\frac{1}{\lambda}$, we can also assume that $d_{\tilde{g}}(\tilde{x},\tilde{y})\gtrsim\frac{1}{r}$.\par

By Lemma \ref{lemma4.3}, $\int_0^{2\pi}e^{iw\cdot\Theta}d\theta=\sqrt{2\pi} |w|^{-1/2}\sum_{\pm}e^{\pm i|w|}+O(|w|^{-3/2}),|w|\geq1$, where $w=r\Phi(\tilde{x},\tilde{y})$. Integrate up $\theta$, the above quantity is then controlled by

\begin{equation}
\begin{split}
& |\sum_\pm\int_{c_1\lambda}^{c_2\lambda}\tilde{\psi}(\lambda-r)|rd_{\tilde{g}}(\tilde{x},\tilde{y})|^{-1/2}e^{\pm ird_{\tilde{g}}(\tilde{x},\tilde{y})}rdr+\int_{c_1\lambda}^{c_2\lambda}\tilde{\psi}(\lambda-r)|rd_{\tilde{g}}(\tilde{x},\tilde{y})|^{-3/2}rdr|\\
\leq & d_{\tilde{g}}(\tilde{x},\tilde{y})^{-1/2}\int_{c_1\lambda}^{c_2\lambda}\tilde{\psi}(\lambda-r)r^{1/2}dr+d_{\tilde{g}}(\tilde{x},\tilde{y})^{-3/2}\int_{c_1\lambda}^{c_2\lambda}\tilde{\psi}(\lambda-r)r^{-1/2}dr\\
= & d_{\tilde{g}}(\tilde{x},\tilde{y})^{-1/2}O(\lambda^{1/2})+O(d_{\tilde{g}}(\tilde{x},\tilde{y})^{-1})\\
= & O((\frac{\lambda}{d_{\tilde{g}}(\tilde{x},\tilde{y})})^{1/2})
\end{split}
\end{equation}
Note that these two equalities are still valid when $c_1$ and $c_2$ are changed to 0 and $\infty$.

Therefore, when $d_{\tilde{g}}(\tilde{x},\tilde{y})\geq\frac{1}{\lambda}$,
\begin{equation}
|\dfrac{w_0(\tilde{x},\tilde{y})}{4\pi^2}\sum_{\pm}\int_{\mathbb{R}^2}\hat{\rho}(\frac{t}{T})e^{i\Phi(\tilde{x},\tilde{y})\cdot\xi\pm
it|\xi|}e^{it\lambda}d\xi|=O((\frac{\lambda}{d_{\tilde{g}}(\tilde{x},\tilde{y})})^{\frac{1}{2}}).
\end{equation}

Now we have finished the estimates for $\alpha=\textit{Id}$. For $\alpha\neq \textit{Id}$, note that we can find a constant $C_p$ that is different from 0, depending on the universal covering, $p$, of the manifold $M$, such that
\begin{equation}\label{3.37}
d_{\tilde{g}}(\tilde{x},\alpha(\tilde{y}))>C_p,
\end{equation}
for all $\alpha\in\Aut$ different from $Id$. The constant $C_p$ comes from the fact that if we assume that the injectivity radius of $M$ is greater than a number, say, 1, and that $x$ is the center of some geodesic ball with radius one contained in $M$, then we can choose the fundamental domain $D$ such that $x$ is at least some distance, say, $C_p>1$, away from any translation of $D$, which we denote as $\alpha(D)$, for any $\alpha\in\Aut$ that is not identity. Therefore, we may use the estimates for $d_{\tilde{g}}(\tilde{x},\tilde{y})\geq\frac{1}{\lambda}$ before, assuming $\lambda$ is larger than $\frac{1}{C_p}$. Use the Hadamard parametrix, (see \cite{SZ}), similarly as before, estimating only the main term,
\begin{equation}\label{3.13}
\begin{split}
& |\int_{-T}^T\hat{\rho}(\frac{t}{T})(\cos t\sqrt{-\Delta_{\tilde{g}}})(\tilde{x},\alpha(\tilde{y}))e^{it\lambda}dt|\\
\lesssim & |(2\pi)^{-2}\int_{\mathbb{R}^2}\int_{-T}^T\hat{\rho}(\frac{t}{T})e^{i\Phi(\tilde{x},\alpha(\tilde{y}))\cdot\xi}\cos(t|\xi|)e^{it\lambda}dt|\\
\lesssim & \sum_{\pm}|\int_0^{2\pi}\int_0^\infty\int^T_{-T}e^{ir\Phi(\tilde{x},\alpha(\tilde{y}))\cdot\Theta\pm itr+it\lambda}\hat{\rho}(\frac{t}{T})rdtdrd\theta|\\
\lesssim & \sum_\pm\int_0^\infty\int_0^{2\pi}\tilde{\psi}(\lambda-r)e^{ir\Phi(\tilde{x},\alpha(\tilde{y}))\cdot\Theta\pm itr+it\lambda}rd\theta dr\\
\leq & \sum_\pm\int_0^\infty\tilde{\psi}(\lambda-r)|rd_{\tilde{g}}(\tilde{x},\alpha(\tilde{y}))|^{-\frac{1}{2}}e^{ir\psi(\tilde{x},\alpha(\tilde{y}))\cdot\Theta\pm itr+it\lambda}rdr+\int_0^\infty\tilde{\psi}(\lambda-r)|rd_{\tilde{g}}(\tilde{x},\alpha(\tilde{y}))|^{-\frac{3}{2}}rdr\\
= & O\Big(\big(\dfrac{\lambda}{d_{\tilde{g}}(\tilde{x},\alpha(\tilde{y}))}\big)^{\frac{1}{2}}\Big).
\end{split}
\end{equation}

Now we have shown all the estimates \eqref{2.9}, \eqref{2.10}, and \eqref{2.11}. Totally, $K(x,y)$ is
\begin{equation}\label{3.33}O(\dfrac{1}{T}(\dfrac{\lambda}{\lambda^{-1}+d_{\tilde{g}}(\tilde{x},\tilde{y})})^{\frac{1}{2}})+\sum_{Id\neq\alpha\in \Aut}[O(\dfrac{1}{T}(\dfrac{\lambda}{d_{\tilde{g}}(\tilde{x},\alpha(\tilde{y}))})^{1/2})+O(\dfrac{e^{ET}}{T})],
\end{equation}
where $E=\max\{C_N,d\}+1$.

Note that, by the finite propagation speed of the wave operator $\partial^2_t-\Delta_{\tilde{g}}$, $d_{\tilde{g}}(\tilde{x},\alpha(\tilde{y}))\leq T$ in the support of $\cos(t\sqrt{-\Delta_g})(\tilde{x},\alpha(\tilde{y}))$. While $M$ is a compact manifold with nonpositive curvature, the number of terms of $\alpha$'s such that $d_{\tilde{g}}(\tilde{x},\alpha(\tilde{y}))\leq T$ is at most $e^{cT}$\footnote{The number of terms of $\alpha$'s such that $d_{\tilde{g}}(\tilde{x},\alpha(\tilde{y}))\leq T$ is also bounded below by $e^{c'T}$ for some constant $c'$ depending on the curvature of the manifold, according to G\"unther and Bishop's Comparison Theorem in \cite{IS} (also see \cite{SZ}).}, for some constant $c$ depending on the curvature, by the Bishop Comparison Theorem (see \cite{IS}\cite{SZ}).

We take the $L^r(\gamma)$ norms of each individual terms first, then by the Minkowski's inequality, $||K(x,\cdot)||_{L^r(\gamma[-1,1])}$ is bounded by the sum. Also note that we may consider the geodesic distance to be comparable to the arc length of the geodesic.

The first term is simple, and it is controlled by a constant times
\begin{equation}
\dfrac{1}{T}(\int_0^1(\dfrac{\lambda}{\lambda^{-1}+\tau})^{\frac{r}{2}}d\tau)^{1/r}=O(\dfrac{\lambda^{\frac{p-2}{p}}}{T}).
\end{equation}

Accounting in the number of terms of those $\alpha$'s, the second term is bounded by a constant times
\begin{equation}
e^{cT}\cdot\dfrac{\lambda^{\frac{1}{2}}}{T}(\int_0^1(\frac{1}{C_p})^{\frac{r}{2}}d\tau)^{\frac{1}{r}}=O(e^{cT}\dfrac{\lambda^{\frac{1}{2}}}{T}).
\end{equation}

Therefore,
\begin{equation}\label{3.17}
\begin{split}
||K(x,\cdot)||_{L^r(\gamma[-1,1])}= & O(\dfrac{\lambda^{\frac{p-2}{p}}}{T})+O(e^{cT}\frac{\lambda^{\frac{1}{2}}}{T})+O(\dfrac{e^{(c+E)T}}{T})\\
= & I+II+III.
\end{split}
\end{equation}

Now take $T=\beta\log\lambda$, where $\beta\leq\dfrac{p-4}{2(c+E)p}$. (Note that we can assume that $c\neq0$, otherwise, there is only one $\alpha$ that we are considering, which is $\alpha=Id$.) Then
\begin{equation}
I=II=O(\dfrac{\lambda^{\frac{p-2}{p}}}{\log\lambda}),
\end{equation}
and
\begin{equation}
III=o(\dfrac{\lambda^{\frac{p-2}{p}}}{\log\lambda}).
\end{equation}

Summing up, we get that
\begin{equation}
||K(x,\cdot)||_{L^r(\gamma[-1,1])}=O\big(\frac{\lambda^{\frac{p-2}{p}}}{\log\lambda}\big).
\end{equation}

Now apply Young's inequality, with $r=\frac{p}{2}$, we get that $$\forall f\in L^{p'}(\gamma),||\chi^\lambda_T (\chi^\lambda_T)^* f||_{L^p(\gamma)}\lesssim\frac{(1+\lambda)^{1-\frac{2}{p}}}{\log\lambda}||f||_{L^{p'}(\gamma)}.$$ Therefore, Theorem~\ref{theorem2} is proved.

\newsection{Higher dimensions, $n\geq3$}

Now we move on to the case for $n\geq3$. While we want to show Theorem \ref{theorem1} for the full range of $p$ directly, we can only show it under the condition that $p>\frac{4k}{n-1}$ using the same method as in the last section. Although we only need $p=\infty$ later to interpolate and get to the full version of Theorem \ref{theorem1}, we will show the most as we can for the moment.

\begin{theorem}\label{theorem4.1}
Let $(M,g)$ be a  compact smooth $n$-dimensional boudaryless Riemannian manifold with nonpositive curvature, and $\Sigma$ be an $k$-dimensional compact smooth submanifold on $M$, then for any $f\in L^2(M)$, we have the following estimate
\begin{equation}\label{4.1}
||\sum_{|\lambda_j-\lambda|\leq(\log\lambda)^{-1}}E_jf||_{L^p(\Sigma)}\lesssim\frac{\lambda^{\delta(p)}}{(\log \lambda)^{\frac{1}{2}}}||f||_{L^2(M)},\ \ \ \forall p>\frac{4k}{n-1},
\end{equation}
where
\begin{equation}
\delta(p)=\frac{n-1}{2}-\frac{k}{p}.
\end{equation}
\end{theorem}

\begin{remark}
Note that although this estimate is not complete (that works for all $p>2$) for general numbers $k<n$, we get the complete range of $p\geq2$ when $k$ and $n$ satisfy $\frac{4k}{n-1}<2$. That means that we get the improvement for all $p\geq2$ when $k=1$, $n>3$; $k=2$, $n>5$; etc..
\end{remark}

For $n\geq3$, for the sake of using interpolation later, we need to insert a bump function\footnote{We do not need the bump function if we simply want to prove Theorem \ref{theorem4.1}.}. Take $\varphi\in C_0^\infty(\mathbb{R})$ such that $\varphi(t)=1$ when $|t|\leq\frac{1}{2}$ and $\varphi(t)=0$ when $|t|>1$. Then we only have to consider the following kernel\footnote{This kernel is different from the one in \eqref{2.5}.}
\begin{equation}
K(x,y)=\frac{1}{\pi T}\int_{-T}^T(1-\varphi(t))\hat{\rho}(\frac{t}{T})(\cos t\sqrt{-\Delta_g})(x,y))e^{it\lambda}dt,
\end{equation}
which is non-zero only when $|t|>\frac{1}{2}$. In the following discussion, we may sometimes only show estimates for $K(x,y)$ when $t>\frac{1}{2}$, as the part for $t<-\frac{1}{2}$ can be done similarly.

The reason why we only consider the above kernel $K(x,y)$ is because of the following lemma.
\begin{lemma}\label{lemma4.4}
For $\varphi\in C_0^\infty(\mathbb{R})$ such that $\varphi(t)=1$ when $|t|\leq\frac{1}{2}$ and $\varphi(t)=0$ when $|t|>1$. Let
\begin{equation}
\tilde{K}(x,y)=\frac{1}{\pi T}\int_{-1}^1\varphi(t)\hat{\rho}(\frac{t}{T})(\cos t\sqrt{-\Delta_g})(x,y)e^{it\lambda}dt,
\end{equation}
then
\begin{equation}
\sup_x||\tilde{K}(x,\cdot)||_{L^r(\Sigma)}=O(\dfrac{\lambda^{2\delta(p)}}{\log\lambda}).
\end{equation}
\end{lemma}
We will postpone the proof to the end of this section.

Now we are ready to prove Theorem \ref{theorem4.1}, which is essentially the same as the lower dimension case, and what we need to show is \eqref{2.6}. By a partition of unity, we may choose some point $x\in\Sigma$, and consider $\Sigma$ to be within a ball with geodesic radius 1 centered at $x$, and under the geodesic normal coordinates centered at $x$, parametrize $\Sigma$ as
\begin{equation}\nonumber
\Sigma=\{(t,\Theta)|y=\exp_x(t\Theta)\in\Sigma,t\in[-1,1], \Theta\in\mathbb{S}^{k-1}\}
\end{equation}

Applying the Hadamard Parametrix, for any $\alpha\in\Aut$,
\begin{equation}\label{4.9}
\cos(t\sqrt{-\Delta_{\tilde{g}}})(\tilde{x},\alpha(\tilde{y}))=\dfrac{w_0(\tilde{x},\alpha(\tilde{y}))}{(2\pi)^n}\sum_{\pm}\int_{\mathbb{R}^n}e^{i\Phi(\tilde{x},\alpha(\tilde{y}))\cdot\xi\pm
it|\xi|}d\xi+\sum_{\nu=1}^\infty w_\nu(\tilde{x},\tilde{y})\mathcal{E}_\nu(t,d_{\tilde{g}}(\tilde{x},\alpha(\tilde{y})))+R_N(t,\tilde{x},\alpha(\tilde{y})),
\end{equation}
where
$|\Phi(\tilde{x},\alpha(\tilde{y}))|=d_{\tilde{g}}(\tilde{x},\alpha(\tilde{y}))$, and $\mathcal{E}_\nu,\nu=1,2,3,...$ are those described in Section 3.

By Theorem \ref{theorem3.3},
\begin{equation}
\int_{-T}^T|R_N(t,\tilde{x},\alpha(\tilde{y}))|dt\lesssim\int_0^Te^{dt}dt=O(e^{dT}).
\end{equation}
Moreover, by \eqref{3.22}, for $\nu=1,2,3,...$,
\begin{equation}|\int_{-T}^T(1-\varphi(t))\hat{\rho}(\frac{t}{T})e^{it\lambda}\mathcal{E}_\nu(t,d_{\tilde{g}}(\tilde{x},\alpha(\tilde{y})))dt|=O(T(T\lambda)^{n-1-2\nu}).
\end{equation}

Since $|w_\nu(\tilde{x},\alpha(\tilde{y}))|=O(\exp(c_\nu d_{\tilde{g}}(\tilde{x},\alpha(\tilde{y}))))$ by \cite{Berard}, for some constant $c_\nu$ depending on $\nu$,
\begin{equation}
\begin{split}
&\sum_{\nu=1}^N |w_\nu(\tilde{x},\alpha(\tilde{y}))\int_{-T}^T(1-\varphi(t))\hat{\rho}(\frac{t}{T})e^{it\lambda}\mathcal{E}_\nu(t,d_{\tilde{g}}(\tilde{x},\alpha(\tilde{y})))dt|\\
=&\sum_{\nu=1}^N O(T(T\lambda)^{n-1-2\nu}\exp(c_\nu d_{\tilde{g}}(\tilde{x},\alpha(\tilde{y}))))\\
=&O(T^{n-2}\lambda^{n-3}\exp(C_N d_{\tilde{g}}(\tilde{x},\alpha(\tilde{y})))),
\end{split}
\end{equation}
for some $C_N$ depending on $c_1, c_2,..., c_{N-1}$.

All in all, disregarding the integral of the remainder kernel,
\begin{multline}
|\int_{-T}^T(1-\varphi(t))\hat{\rho}(\frac{t}{T})\cos(t\sqrt{-\Delta_{\tilde{g}}})(\tilde{x},\alpha(\tilde{y}))e^{it\lambda}dt|\\
=|\int_{-T}^T(1-\varphi(t))\hat{\rho}(\frac{t}{T})\dfrac{w_0(\tilde{x},\tilde{y})}{(2\pi)^n}\sum_{\pm}\int_{\mathbb{R}^2}e^{i\Phi(\tilde{x},\alpha(\tilde{y}))\cdot\xi\pm
it|\xi|}e^{it\lambda}d\xi dt|+O(T^{n-2}\lambda^{n-3}\exp(C_N d_{\tilde{g}}(\tilde{x},\alpha(\tilde{y})))).
\end{multline}

On the other hand, $|w_0(\tilde{x},\tilde{y})|=O(1)$ (see \cite{SZ}) by applying G\"unther's Comparison Theorem in \cite{Gu}, and for
\begin{equation}\label{4.13}
|\sum_{\pm}\int_{\mathbb{R}^n}\int_{-T}^T(1-\varphi(t))e^{i\Phi(\tilde{x},\alpha(\tilde{y}))\cdot\xi\pm it|\xi|+it\lambda}\hat{\rho}(\frac{t}{T})dtd\xi|,
\end{equation}
as we may assume as before that $d_{\tilde{g}}(\tilde{x},\alpha(\tilde{y}))>\frac{1}{2}$ by the stationary phase estimates in \cite{soggebook}.

Denote that $\psi(t)=(1-\varphi(t))\hat{\rho}(\frac{t}{T})$, and $\tilde{\psi}$ is the inverse Fourier Transform of $\psi$.

Again we have, $\tilde{\psi}(\lambda+|\xi|)=O(T(1+\lambda+|\xi|)^{-N})$, for any $N\in\mathbb{N}$, so $\int_{\mathbb{R}^n}\tilde{\psi}(\lambda+|\xi|)d\xi$ can be arbitrarily small, while $\tilde{\psi}(\lambda-|\xi|)=O(T(1+T|\lambda-|\xi||)^{-N})$.

Integrate \eqref{4.13} with respect to $t$ first, then it is bounded by a constant times
\begin{equation}
\sum_\pm\int_0^\infty\int_{\mathbb{S}^{n-1}}\tilde{\psi}(\lambda\pm r)e^{ir\Phi(\tilde{x},\alpha(\tilde{y}))\cdot\Theta}r^{n-1}d\Theta dr.
\end{equation}

Because $\tilde{\psi}(\lambda\pm r)\leq T(1+T|\lambda\pm r|)^{-N}$ for any $N>0$, the term with $\tilde{\psi}(\lambda+r)$ in the sum is $O(1)$, while the other term with $\tilde{\psi}(\lambda-r)$ is significant only when $r$ is comparable to $\lambda$, say, $c_1\lambda<r<c_2\lambda$ for some constants $c_1$ and $c_2$. In this case, as we assumed that $d_{\tilde{g}}(\tilde{x},\alpha(\tilde{y}))\geq D$, we can also assume that $d_{\tilde{g}}(\tilde{x},\alpha(\tilde{y}))\gtrsim\frac{1}{r}$ for large $\lambda$.

By Lemma \ref{lemma4.3}, $\int_{\mathbb{S}^{n-1}}e^{iw\cdot\Theta}d\Theta=\sqrt{2\pi}^{n-1} |w|^{-\frac{n-1}{2}}\sum_{\pm}e^{\pm i|w|}+O(|w|^{-\frac{n+1}{2}}),|w|\geq1$, where $w=r\Phi(\tilde{x},\alpha(\tilde{y}))$. Integrate up $\Theta$, the above quantity is then controlled by

\begin{equation}
\begin{split}
& |\sum_\pm\int_{c_1\lambda}^{c_2\lambda}\tilde{\psi}(\lambda-r)|rd_{\tilde{g}}(\tilde{x},\alpha(\tilde{y}))|^{-\frac{n-1}{2}}e^{\pm ird_{\tilde{g}}(\tilde{x},\tilde{y})}r^{n-1}dr+\int_{c_1\lambda}^{c_2\lambda}\tilde{\psi}(\lambda-r)|rd_{\tilde{g}}(\tilde{x},\alpha(\tilde{y}))|^{-\frac{n+1}{2}}r^{n-1}dr|\\
\leq & d_{\tilde{g}}(x,y)^{-\frac{n-1}{2}}\int_{c_1\lambda}^{c_2\lambda}\tilde{\psi}(\lambda-r)r^{\frac{n-1}{2}}dr+d_{\tilde{g}}(\tilde{x},\alpha(\tilde{y}))^{-\frac{n+1}{2}}\int_{c_1\lambda}^{c_2\lambda}\tilde{\psi}(\lambda-r)r^{\frac{n-3}{2}}dr\\
= & O((\frac{\lambda}{d_{\tilde{g}}(\tilde{x},\alpha(\tilde{y}))})^{\frac{n-1}{2}})
\end{split}
\end{equation}

Therefore, disregarding the integral of the remainder kernel,
\begin{equation}
\int_{-T}^T(1-\varphi(t))\hat{\rho}(\frac{t}{T})(\cos t\sqrt{-\Delta_{\tilde{g}}})(\tilde{x},\alpha(\tilde{y}))e^{it\lambda}dt=O((\frac{\lambda}{d_{\tilde{g}}(\tilde{x},\alpha(\tilde{y}))})^{\frac{n-1}{2}})+O(T^{n-2}\lambda^{n-3}\exp(C_N d_{\tilde{g}}(\tilde{x},\alpha(\tilde{y})))).
\end{equation}

Now $K(x,y)$ is
\begin{equation}\label{4.16}
 \sum_{\alpha\in \Aut}[O(\frac{1}{T}(\dfrac{\lambda}{d_{\tilde{g}}(\tilde{x},\alpha(\tilde{y}))})^{\frac{n-1}{2}})+O(\dfrac{e^{ET}}{T})],
\end{equation}
where $E=\max\{C_N,d\}+1$.

Here we still have: the number of terms of $\alpha$'s such that $d_{\tilde{g}}(\tilde{x},\alpha(\tilde{y}))\leq T$ is at most $e^{cT}$, for some constant $c$ depending on the curvature, and there exists a constant $C_p$ such that $d_{\tilde{g}}(\tilde{x},\alpha(\tilde{y}))>C_p$ for any $\alpha\in\Aut$ different from identity.

Now we take the $L^r(\Sigma)$ norms of each individual terms. By \eqref{3.37}, and accounting in the number of terms of those $\alpha$'s, the first one is bounded by a constant times
\begin{equation}
\dfrac{e^{cT}\lambda^{\frac{n-1}{2}}}{T}(\int_0^1C_p^{-\frac{n-1}{2}\cdot r}\tau^{k-1}d\tau)^{\frac{1}{r}}=O(\dfrac{e^{cT}\lambda^{\frac{n-1}{2}}}{T}).
\end{equation}

Therefore,
\begin{equation}
\begin{split}
||K(x,\cdot)||_{L^r(\Sigma)}
= &O(\dfrac{e^{cT}\lambda^{\frac{n-1}{2}}}{T})+O(\dfrac{e^{(c+E)T}}{T})\\
= & I+II.
\end{split}
\end{equation}

Now take $T=\beta\log\lambda$, where $\beta=\dfrac{\frac{n-1}{2}-\frac{2k}{p}-\delta}{c+E}$, where $\delta$ satisfies $0<\delta<\frac{n-1}{2}-\frac{2k}{p}$. Note that $\frac{n-1}{2}-\frac{2k}{p}>0$ when $p>\frac{4k}{n-1}$. Then
\begin{equation}
I=O(\dfrac{\lambda^{\beta c+\frac{n-1}{2}}}{\log\lambda})=O(\dfrac{\lambda^{\frac{n-1}{2}-\frac{2k}{p}-\delta+\frac{n-1}{2}}}{\log\lambda})=o(\dfrac{\lambda^{n-1-\frac{2k}{p}}}{\log\lambda}),
\end{equation}
and
\begin{equation}
II=O(\dfrac{\lambda^{\beta(c+E)}}{\log\lambda})=O(\dfrac{\lambda^{\frac{n-1}{2}-\frac{2k}{p}-\delta}}{\log\lambda})=o(\frac{\lambda^{n-1-\frac{2k}{p}}}{\log\lambda}).
\end{equation}

Summing up, we get that
\begin{equation}
||K(x,\cdot)||_{L^r(\Sigma)}=o\big(\dfrac{\lambda^{n-1-\frac{2k}{p}}}{\log\lambda}\big).
\end{equation}

Now apply Young's inequality, with $r=\frac{p}{2}$, together with the estimate in Lemma \ref{lemma4.4}, we have
\begin{equation}
\forall f\in L^{p'}(\Sigma),||\chi^\lambda_T (\chi^\lambda_T)^* f||_{L^p(\Sigma)}\lesssim\frac{\lambda^{n-1-\frac{2k}{p}}}{\log\lambda}||f||_{L^{p'}(\Sigma)}.
\end{equation}
Therefore, Theorem~\ref{theorem4.1} is proved.

\begin{proof}[proof of Lemma \ref{lemma4.4}]
With similar approaches as the previous discussions, we can show that $\tilde{K}(x,y)$ is
\begin{equation}\label{4.30}O(\dfrac{1}{T}(\dfrac{\lambda}{\lambda^{-1}+d_{\tilde{g}}(\tilde{x},\tilde{y})})^{\frac{n-1}{2}})
 +\sum_{Id\neq\alpha\in \Aut}[O(\dfrac{1}{T}(\dfrac{\lambda}{d_{\tilde{g}}(\tilde{x},\alpha(\tilde{y}))})^{\frac{n-1}{2}})+O(e^{ET})],
\end{equation}
where $E=\max\{C_N,d\}+1$.

Note that $|t|\leq1$ for $\varphi(t)\neq0$, and the number of terms such that $d_{\tilde{g}}(\tilde{x},\alpha(\tilde{y}))\leq1$ is at most $e^c$, so that
\begin{equation}
||\tilde{K}(x,y)||_{L^r(\Sigma)}=O(\dfrac{\lambda^{2\delta(p)}}{\log\lambda}),
\end{equation}
if we take $T=\log\lambda$ and calculate as before.
\end{proof}

\newsection{Proof of the main theorem in all dimensions}

To show Theorem \ref{theorem1}, we need to use interpolation. Recall that
\begin{equation}
\begin{split}
K(x,y)= & \frac{1}{\pi T}\int_{-T}^T(1-\varphi(t))\hat{\rho}(\frac{t}{T})(\cos t\sqrt{-\Delta_g})(x,y))e^{it\lambda}dt\\
= & \frac{1}{2\pi T}\int_{-T}^T(1-\varphi(t))\hat{\rho}(\frac{t}{T})(e^{it\sqrt{-\Delta_g}}+e^{-it\sqrt{-\Delta_g}})(x,y)e^{it\lambda}dt
\end{split}
\end{equation}
is the kernel of the operator
\begin{equation}
\begin{split}
& \dfrac{1}{2\pi T}[\sum_j\tilde{\psi}(\lambda-\lambda_j)E_j+\sum_j\tilde{\psi}(\lambda+\lambda_j)E_j]\\
= & \dfrac{1}{2\pi T}[\sum_j\tilde{\psi}(\lambda-\lambda_j)E_j]+O(1)\\
= & \dfrac{1}{2\pi T}\tilde{\psi}(\lambda-\sqrt{-\Delta_g})+O(1),
\end{split}
\end{equation}
where $\tilde{\psi}(t)$ is the inverse Fourier transform of $(1-\varphi(t))\hat{\rho}(\frac{t}{T})$ so that $|\tilde{\psi}(t)|\leq T(1+|t|)^{-N}$ for any $N\in\mathbb{N}$.

We have the following estimate for $\tilde{\psi}(\lambda-\sqrt{-\Delta_g})$.

\bth
For $k\neq n-2$,
\begin{equation}\label{4.28}
||\tilde{\psi}(\lambda-P)g||_{L^2(\Sigma)}\lesssim T\lambda^{2\delta(2)}||g||_{L^2(\Sigma)},\ \ \text{for any}\ g\in L^2(\Sigma),
\end{equation}
and for $k=n-2$,
\begin{equation}
||\tilde{\psi}(\lambda-P)g||_{L^2(\Sigma)}\lesssim T\lambda^{2\delta(2)}\log\lambda||g||_{L^2(\Sigma)},\ \ \text{for any}\ g\in L^2(\Sigma),
\end{equation}
where $P=\sqrt{-\Delta_g}$.
\end{theorem}
\begin{proof}
Recall the proof of the corresponding restriction theorem in \cite{burq}, they showed that for $\chi\in\mathcal{S}(\mathbb{R})$, and define
\begin{equation}
\chi_\lambda=\chi(\sqrt{-\Delta_g}-\lambda)=\sum_j\chi(\lambda_j-\lambda)E_j,
\end{equation}
we have
\begin{equation}
||\chi_\lambda||_{L^2(M)\rightarrow L^2(\Sigma)}=O(\lambda^{\delta(2)}),
\end{equation}
for $k\neq n-2$,
and
\begin{equation}
||\chi_\lambda||_{L^2(M)\rightarrow L^2(\Sigma)}=O(\lambda^{\delta(2)}(\log\lambda)^{\frac{1}{2}}),
\end{equation}
for $k=n-2$.

Now consider  $\tilde{\psi}(\lambda-P)$ as $S\tilde{S}^*$, where
\begin{equation}
S=\sum_j(1+|\lambda_j-\lambda|)^{-M}E_j
\end{equation}
and
\begin{equation}
\tilde{S}=\sum_j(1+|\lambda_j-\lambda|)^M\tilde{\psi}(\lambda_j-\lambda)E_j,
\end{equation}
where $M$ is some large number.

Recall that $|\tilde{\psi}(\tau)|\leq T(1+|\tau|)^{-N}$ for any $N\in\mathbb{N}$, we then have
\begin{equation}
|(1+|\lambda_j-\lambda|)^M\tilde{\psi}(\lambda_j-\lambda)|\leq T(1+|\lambda_j-\lambda|)^{-N}
\end{equation}
for any $N$.

By \eqref{1.7}, which we deduced from the proof of Theorem 3 in \cite{burq}, for a given $\lambda$,
\begin{equation}
||\sum_{\lambda_j\in(\lambda-1,\lambda+1)}E_j||_{L^2(M)\rightarrow L^2(\Sigma)}=O(\lambda^{\delta(2)}),\ \ \ \ \text{if}\ k\neq n-2
\end{equation}
and
\begin{equation}
||\sum_{\lambda_j\in(\lambda-1,\lambda+1)}E_j||_{L^2(M)\rightarrow L^2(\Sigma)}=O(\lambda^{\delta(2)}(\log\lambda)^{\frac{1}{2}}),\ \ \ \ \ \text{if}\ k=n-2
\end{equation}
so that for any $f\in L^2(M)$,
\begin{equation}
\begin{split}
&||\sum_j(1+|\lambda_j-\lambda|^{-M})E_j f||_{L^2(\Sigma)}\\
\leq  &||\sum_{\lambda_j\in(\lambda-1,\lambda+1)}E_jf||_{L^2(\Sigma)}+||\sum_{\lambda_j\not\in(\lambda-\delta,\lambda+\delta)}(1+|\lambda_j-\lambda|^{-M})E_jf||_{L^2(\Sigma)}\\
 \lesssim &\begin{cases}\lambda^{\delta(2)}||f||_{L^2(M)}+\sum_{\lambda_j\not\in(\lambda-1,\lambda+1)}(1+|\lambda_j-\lambda|)^{-M}||E_jf||_{L^2(\Sigma)},\ &\mbox{if}\ k\neq n-2,\\
\lambda^{\delta(2)}(\log\lambda)^{\frac{1}{2}}||f||_{L^2(M)}+\sum_{\lambda_j\not\in(\lambda-1,\lambda+1)}(1+|\lambda_j-\lambda|)^{-M}||E_jf||_{L^2(\Sigma)},\ &\mbox{if}\ k= n-2.
\end{cases}
\end{split}
\end{equation}
As
\begin{equation}
\begin{split}
&\sum_{\lambda_j\not\in(\lambda-1,\lambda+1)}(1+|\lambda_j-\lambda|)^{-M}||E_jf||_{L^2(\Sigma)}\\
\leq
&\begin{cases}\sum_{\lambda_j\not\in(\lambda-1,\lambda+1)}\lambda_j^{\delta(2)}(1+|\lambda_j-\lambda|)^{-M}||E_jf||_{L^2(M)},\ &\mbox{if}\ k\neq n-2,\\
\sum_{\lambda_j\not\in(\lambda-1,\lambda+1)}\lambda_j^{\delta(2)}(\log\lambda_j)^{\frac{1}{2}}(1+|\lambda_j-\lambda|)^{-M}||E_jf||_{L^2(M)},\ &\mbox{if}\ k=n-2,
\end{cases}
\end{split}
\end{equation}
which can be made arbitrarily small when $M$ is sufficiently large,
\begin{equation}
||\sum_j(1+|\lambda_j-\lambda|^{-M})E_j f||_{L^2(\Sigma)}\leq  \begin{cases}\lambda^{\delta(2)}||f||_{L^2(M)},\ &\mbox{if}\ k\neq n-2,\\
\lambda^{\delta(2)}(\log\lambda)^{\frac{1}{2}}||f||_{L^2(M)},\ &\mbox{if}\ k= n-2.
\end{cases}
\end{equation}

Similarly, we have
\begin{equation}
||\sum_j(1+|\lambda_j-\lambda|^{M})\tilde{\phi}(\lambda_j-\lambda)E_j f||_{L^2(\Sigma)}\leq  \begin{cases}T\lambda^{\delta(2)}||f||_{L^2(M)},\ &\mbox{if}\ k\neq n-2,\\
T\lambda^{\delta(2)}(\log\lambda)^{\frac{1}{2}}||f||_{L^2(M)},\ &\mbox{if}\ k= n-2.
\end{cases}
\end{equation}

Therefore,
\begin{equation}
\begin{split}
||\tilde{\psi}(\lambda-P)g||_{L^2(\Sigma)}= & ||S\tilde{S}^*g||_{L^2(\Sigma)}\\
\leq & ||S||_{L^2(M)\rightarrow L^2(\Sigma)}||\tilde{S}^*||_{L^2(\Sigma)\rightarrow L^2(M)}||g||_{L^2(\Sigma)}\\
= & ||S||_{L^2(M)\rightarrow L^2(\Sigma)}||\tilde{S}||_{L^2(M)\rightarrow L^2(\Sigma)}||g||_{L^2(\Sigma)}\\
\lesssim & \begin{cases}T\lambda^{2\delta(2)}||g||_{L^2(\Sigma)},\ \ \ \ &\mbox{if}\ k\neq n-2,\\
T\lambda^{2\delta(2)}\log\lambda||g||_{L^2(\Sigma)},\ \ \ \ &\mbox{if}\ k=n-2.
\end{cases}
\end{split}
\end{equation}
\end{proof}

Now we may finish the proof of Theorem \ref{theorem1}.

Recall that we denote $K$ as the operator whose kernel is $K(x,y)$. The above theorem tells us that,
\begin{equation}
||K||_{L^2(\Sigma)\rightarrow L^2(\Sigma)}\leq \begin{cases}O(\lambda^{2\delta(2)}), &\text{for}\ k\neq n-2;\\
O(\lambda^{2\delta(2)}\log\lambda), &\text{for}\ k=n-2.
\end{cases}
\end{equation}

Interpolating this with
\begin{equation}
||K||_{L^1(\Sigma)\rightarrow L^\infty(\Sigma)}=O(\dfrac{e^{cT}\lambda^{\frac{n-1}{2}}}{T})
\end{equation}
or
\begin{equation}
||K||_{L^1(\Sigma)\rightarrow L^\infty(\Sigma)}=O(e^{cT}\lambda^{\frac{n-1}{2}})
\end{equation}
respectively by Theorem \ref{theorem4.1}, we get that for any $p$ and $k\neq n-2$,
\begin{equation}
||K||_{L^{p'}(\Sigma)\rightarrow L^p(\Sigma)}=O(\dfrac{\lambda^{\frac{n-1}{2}(1-\frac{2}{p})}e^{cT(1-\frac{2}{p})}\lambda^{2\delta(2)\cdot\frac{2}{p}}}{T^{1-\frac{2}{p}}})=O(\dfrac{\lambda^{\frac{n-1}{2}-\frac{n-1}{p}+\frac{4\delta(2)}{p}}e^{cT(1-\frac{2}{p})}}{T^{1-\frac{2}{p}}}),
\end{equation}
and for $k=n-2$,
\begin{equation}
||K||_{L^{p'}(\Sigma)\rightarrow L^p(\Sigma)}=O(\dfrac{\lambda^{\frac{n-1}{2}-\frac{n-1}{p}+\frac{4\delta(2)}{p}}e^{cT(1-\frac{2}{p})}T^{\frac{2}{p}}}{T^{1-\frac{2}{p}}})=O(\dfrac{\lambda^{\frac{n-1}{2}-\frac{n-1}{p}+\frac{4\delta(2)}{p}}e^{cT(1-\frac{2}{p})}}{T^{1-\frac{4}{p}}}).
\end{equation}

If $k=n-1$, then $\delta(2)=\frac{1}{4}$.
\begin{equation}
||K||_{L^{p'}(\Sigma)\rightarrow L^p(\Sigma)}=O(\dfrac{\lambda^{\frac{n-1}{2}-\frac{n-2}{p}}e^{cT(1-\frac{2}{p})}}{T^{1-\frac{2}{p}}}).
\end{equation}
Since $\frac{n-1}{2}-\frac{n-2}{p}<2\delta(p)$ if $p>\frac{2n}{n-1}$, say, $\frac{n-1}{2}-\frac{n-2}{p}+\delta<2\delta(p)$ for some small number $\delta>0$, then taking $\beta=\frac{\delta}{c(1-\frac{2}{p})}$, and $T=\beta\log\lambda$, we have
\begin{equation}
||K||_{L^{p'}(\Sigma)\rightarrow L^p(\Sigma)}=O(\dfrac{\lambda^{2\delta(p)-\delta}}{T^{1-\frac{2}{p}}})=O(\dfrac{\lambda^{2\delta(p)-\delta}}{(\log\lambda)^{1-\frac{2}{p}}})= o(\dfrac{\lambda^{2\delta(p)}}{\log\lambda}),
\end{equation}
which indicates Theorem \ref{theorem1}.

If $k=n-2$,
\begin{equation}
||K||_{L^{p'}(\Sigma)\rightarrow L^p(\Sigma)}=O(\dfrac{\lambda^{\frac{n-1}{2}-\frac{n-1}{p}+\frac{4\delta(2)}{p}}e^{cT(1-\frac{2}{p})}}{T^{1-\frac{4}{p}}})=O(\dfrac{\lambda^{\frac{n-1}{2}-\frac{n-1}{p}+\frac{2}{p}}e^{cT(1-\frac{2}{p})}}{T^{1-\frac{4}{p}}}).
\end{equation}
Now since $\frac{n-1}{2}-\frac{n-1}{p}+\frac{2}{p}<(n-1)-\frac{2(n-2)}{p}$ when $p>2$, we can take $\delta>0$ such that $\frac{n-1}{2}-\frac{n-1}{p}+\frac{2}{p}+\delta<(n-1)-\frac{2(n-2)}{p}$, and take $\beta=\frac{\delta}{c(1-\frac{2}{p})}$, $T=\beta\log\lambda$, then
\begin{equation}
||K||_{L^{p'}(\Sigma)\rightarrow L^p(\Sigma)}=O(\dfrac{\lambda^{2\delta(p)-\delta}}{(\log\lambda)^{1-\frac{4}{p}}})=o(\dfrac{\lambda^{2\delta(p)}}{\log\lambda}),
\end{equation}
which is the what we need.

If $k\leq n-3$, $\delta(2)=\frac{n-1}{2}-\frac{k}{2}$, then
\begin{equation}
||K||_{L^{p'}(\Sigma)\rightarrow L^p(\Sigma)}=O(\dfrac{\lambda^{\frac{n-1}{2}-\frac{n-1}{p}+\frac{4\delta(2)}{p}}e^{cT(1-\frac{2}{p})}}{T^{1-\frac{2}{p}}})=O(\dfrac{\lambda^{\frac{n-1}{2}-\frac{n-1}{p}+\frac{2(n-1)-2k}{p}}e^{cT(1-\frac{2}{p})}}{T^{1-\frac{2}{p}}}).
\end{equation}
Since $\frac{n-1}{2}-\frac{n-1}{p}+\frac{2(n-1)-2k}{p}<(n-1)-\frac{2k}{p}=2\delta(p)$ for $p>2$, we can take $\delta>0$ such that $\frac{n-1}{2}-\frac{n-1}{p}+\frac{2(n-1)-2k}{p}+\delta<(n-1)-\frac{2k}{p}$, and take $\beta=\frac{\delta}{c(1-\frac{2}{p})}$, $T=\beta\log\lambda$, then
\begin{equation}
||K||_{L^{p'}(\Sigma)\rightarrow L^p(\Sigma)}=O(\dfrac{\lambda^{2\delta(p)-\delta}}{(\log\lambda)^{1-\frac{2}{p}}})=o(\frac{\lambda^{2\delta(p)}}{\log\lambda}),
\end{equation}
which finishes Theorem \ref{theorem1}.


\end{document}